\documentclass[3p,times]{elsarticle}

\usepackage{mathpazo}
\usepackage{amssymb, amsmath,amsthm, mathrsfs}
\usepackage{color}
\usepackage[colorlinks=true, allcolors=blue]{hyperref}
\usepackage{mathtools}
\usepackage{cancel}

\newtheorem{theorem}{Theorem}

\newtheorem{corollary}[theorem]{Corollary}
\newtheorem{lemma}[theorem]{Lemma}
\newtheorem{example}{Example}
\theoremstyle{definition}

\newcommand{\Z}{\mathbb{Z}}

\newcommand{\Lb}{\mathbf{L}}

\newcommand{\ellb}{\mathbf{\ell}}

\DeclarePairedDelimiterX{\norm}[1]{\lVert}{\rVert}{#1}

\hypersetup{nolinks=true}

\begin{document}

\begin{frontmatter}




\title{Generalizing Geometric Nonwindowed Scattering Transforms on Compact Riemannian Manifolds}


\author[msu]{Albert Chua}
\ead{chuaalbe@msu.edu}
\author[cmse]{Yang Yang}
\ead{yangy5@msu.edu}
\address[msu]{Department of Mathematics, Michigan State University, East Lansing, MI, 48824 USA}
\address[cmse]{Department of Computational Mathematics, Science \&
Engineering, Michigan State University, East Lansing, MI, 48824 USA}
\begin{abstract}
Let $\mathcal{M}$ be a compact, smooth, $n$-dimensional Riemannian manifold without boundary. In this paper, we generalize nonwindowed geometric scattering transforms, which we formulate as $\Lb^q(\mathcal{M})$ norms of a cascade of geometric wavelet transforms and modulus operators. We then provide weighted measures for these operators, prove that these operators are well-defined under specific conditions on the manifold, invariant to the action of isometries, and stable to diffeomorphisms for $\lambda$-bandlimited functions. 
\end{abstract}


\end{frontmatter}

\section{Introduction}
For the purposes of this paper $\mathcal{M}$ will be a compact, smooth, $n$-dimensional Riemannian manifold without boundary contained in $\mathbb{R}^d$, where $d \geq n$ with geodesic distance between two points $x_1, x_2 \in \mathcal{M}$ given by $r(x_1, x_2)$ and Laplace-Beltrami operator denoted as $\Delta$. The notation $\Lb^q(\mathcal{M})$ denotes the set of all functions $f: \mathcal{M} \to \mathbb{R}$ such that $\int_{\mathcal{M}} |f(x)|^q \, d\mu(x) < \infty,$ where $d\mu(x)$ is integration with respect to the Riemannian volume, whose measure is given by $\mu$. We use the notation $\text{Isom}(\mathcal{M}_1, \mathcal{M}_2)$ be the set of isometries between manifolds $\mathcal{M}_1$ and $\mathcal{M}_2$. Lastly, the set of diffeomorphisms on $\mathcal{M}$ will be denoted by $\text{Diff}(\mathcal{M})$, and the maximum placement of $\gamma \in \text{Diff}(\mathcal{M})$ will be given by $\|\gamma\|_\infty := \sup_{x \in \mathcal{M}} r(x, \gamma(x))$.

In recent years, deep convolutional neural networks have shown strong performance on various vision-related tasks \cite{VGG, googlenet, resnet}. However, because of how complex deep convolutional architectures are, it is not entirely clear what mechanisms enable deep convolutional networks to get strong performance on these tasks.

In an effort to better understand the properties of deep convolutional architectures, Mallat proposed the scattering transform \cite{InvariantScatteringNetworks, mallat2012group}, which uses a cascade of unlearned filters and nonlinearities to mimic the behavior of a deep convolutional neural network. Using a specific class of wavelet filters, Mallat found that the scattering transform had many desirable properties for machine learning tasks such as translation invariance and stability to small deformations. Additionally, the scattering transform and its generalizations have been found various applications as a general feature extractor, such as in \cite{QMWavelet1, QMWavelet2, QMWavelet3, allys2019rwst, bruna2013audio, bruna2015intermittent, anden2011multiscale, anden2014deep, anden2019joint, sifre2012combined, sifre2013rotation, oyallon2015deep, oyallon2017scaling}. Authors have also explored extensions of the scattering transform to semi-discrete frames \cite{wiatowski, wiatowski2017energy, koller} and more general gabor frames \cite{czaja, czaja2020rotationally}.

However, certain forms of data, such as point cloud data have a geometry than is non-Euclidean, which motivate manifold learning models \cite{tenenbaum2000global, coifman2006diffusion, van2008visualizing} and geometric deep learning \cite{bronstein2017geometric}. As an extension of the scattering transform, authors have considered graph scattering transforms in \cite{graph, gama, gama2018diffusion, zou2020graph, perlmutter2023asymmetric} via constructing graph wavelets \cite{coifman2006diffusion, hammond2011wavelets}; additionally \cite{saito2023multiscale, saito2024multiscale} extend wavelets and the scattering transform to simplicial complexes; the overarching idea is that these extensions of the scattering transform have similar desirable stability properties, and present success on non-Euclidean datasets.

Extensions of the scattering transform have also been presented on compact Riemannian manifolds via defining the wavelet transform on a compact Riemannian manifold using eigenfunctions of the Laplace-Beltrami operator in \cite{geometric}. In particular, for the windowed geometric scattering transform, the authors were able to prove that the representation was locally invariant to isometries for all $\Lb^2(\mathcal{M})$ functions and stable to deformations in $\text{Diff}(\mathcal{M})$ for $\Lb^2(\mathcal{M})$ under mild restrictions. The authors derive a representation, the nonwindowed geometric scattering transform, which is invariant to isometries and stable to deformations in $\text{Diff}(\mathcal{M})$. 

Regarding nonwindowed geometric scattering transforms, their Euclidean counterparts, nonwindowed scattering transforms, have been effective for applications in quantum chemistry, audio synthesis, and physics have appeared in \cite{QMWavelet1, QMWavelet2, QMWavelet3, allys2019rwst, bruna2013audio}; theoretical results involving stability to deformations have also been provided in \cite{chua2023generalizations}. The main idea behind nonwindowed scattering transforms is that they provide a small number of descriptive features for high dimensional data. As a natural extension, \cite{chew2022manifold} used scattering moments on manifolds for classifications tasks involving point cloud data. However, there were limited theoretical results in \cite{chew2022geometric} for the representation provided in \cite{chew2022manifold}, which motivates this paper.

We provide a well-defined weighted measure for scattering moments when $f \in \Lb^q(\mathcal{M})$ for $q \in (1,2)$, which is an improvement upon requiring $f \in \Lb^2(\mathcal{M})$ . However, our weighted measure is only defined for an aribtrary, finite number of layers, and requires restrictions on the regularity of the manifold. Additionally, we prove a diffeomorphism stability result in the spirit of \cite{mallat2012group, geometric}.

\subsection{Invariance and Stability}
In machine learning tasks, it is often necessary for a representation to have some degree of invariance with respect to the action of a Lie group. For tasks involving manifolds, one may like to have local isometry invariance. More formally, let $V_\xi f(x)= f(\xi^{-1}x)$ for $\xi \in \text{Isom}(\mathcal{M})$ and consider a representation $\Phi: \mathcal{B}_1 \to \mathcal{B}_2$, where $\mathcal{B}_1, \mathcal{B}_2$ are Banach Spaces. It would be desirable to have a representation such that
$$\|\Phi f - \Phi V_\xi f\|_{\mathcal{B}_2} \leq 2^{-dJ} \|\xi\|_\infty \|f\|_{\mathcal{B}_1},$$
where $J$ controls the degree of invariance. A simple choice is to use an averaging filter, but this potentially leads to the loss of information that can be crucial for the task.

For other tasks, such as manifold classification, a fully rigid representation may be required, and full isometry invariance is desirable. That is to say, we have $$\|\Phi f\|_{\mathcal{B}_2} = \|\Phi V_\xi f\|_{\mathcal{B}_2}.$$

In addition to invariance, it is necessary for a representation to also have stability properties. Instead of considering an isometry, consider $\xi \in \text{Diff}(\mathcal{M})$ and think of $V_\xi f$ as a small deformation of $f$. We want
$$\|\Phi f - \Phi V_\xi f\|_{\mathcal{B}_2} \leq 2^{-dJ} K(\xi) \|f\|_{\mathcal{B}_1},$$
where $K(\xi) \to 0$ as $\|\xi\|_\infty \to 0$; this is to ensure that small deformations of an input do not lead to large changes in the representation.

\section{A Review of the Geometric Scattering Transform on Manifolds}

\subsection{Spectral Filters and the Geometric Wavelet Transform}
The convolution of $f,g \in \Lb^2(\mathbb{R}^n)$ is usually defined in space as 
$$(f \ast g)(x) = \int_{\mathbb{R}^n} f(y)g(x-y) \, dy.$$
However, for a general manifold, even under the conditions we have prescribed, a notation of translation does not necessarily exist. Instead, one can consider a spectral definition of convolution via the spectral decomposition of $-\Delta$. Denote $\mathbb{N}_0 := \mathbb{N} \cup\{0\}$. Because our manifold is compact, it is well known that $-\Delta$ has a discrete spectrum, and we can order the eigenvalues in increasing order and denote them as $\{\lambda_n\}_{n \in \mathbb{N}_0}$. We will denote the corresponding eigenfunctions as $\{e_n(x)\}_{n \in \mathbb{N}_0}$, which form an orthonormal basis for $\Lb^2(\mathcal{M})$.

Suppose $f \in \Lb^2(\mathcal{M})$. Since the set of functions $\{e_n(x)\}_{n \in \mathbb{N}_0}$ forms a basis in $\Lb^2(\mathcal{M})$, we decompose 
\begin{equation}
f(x) = \sum_{n \in \mathbb{N}_0} \langle f, e_n \rangle e_n(x) = \sum_{n \in \mathbb{N}_0} \left(\int_{\mathcal{M}} f(y) \overline{e_n(y)}\, d \mu (y)\right) e_n(x),
\end{equation} which is similar to a Fourier series. Since $e_n(y)$ is a replacement for a Fourier mode, it is natural to let 
\begin{equation}
\hat{f}(n) = \int_{\mathcal{M}} f(y) \overline{e_n(y)}\, d \mu(y)
\end{equation} and define convolution on $\mathcal{M}$ between functions $f,h \in \Lb^2(\mathcal{M})$ as
\begin{equation}
f \ast h(x) = \sum_{n \in \mathbb{N}_0} \hat{f}(n) \hat{h}(n) e_n(x).
\end{equation} Defining the operator $T_hf(x) := f \ast h(x)$, it is easy to verify that the kernel for $T_h$ is given by
\begin{equation}
\tilde{K}_h(x,y) := \sum_{n \in \mathbb{N}_0} \hat{h}(n)e_n(x)\overline{e_n(y)}.
\end{equation}
Similar to how convolution commutes with translations on $\mathbb{R}^n$, it is important for convolution on $\mathcal{M}$ to be equivariant to a group action on $\mathcal{M}$. The authors of \cite{geometric} construct an operator by convolving with functions that commute with isometries since the the geometry of $\mathcal{M}$ should be preserved by a representation.

To accomplish this goal, we use a similar definition for spectral filters. A filter $h \in \Lb^2(\mathcal{M})$ is a spectral filter if $\lambda_k = \lambda_\ell$ implies $\hat{h}(k) = \hat{h}(\ell)$. One can prove that there exists $H:[0, \infty) \to \mathbb{R}$ such that 
$$H(\lambda_n) = \hat{h}(n), \qquad \forall n \in \mathbb{N}_0.$$

Let $G : [0, \infty) \to \mathbb{R}$ be nonnegative and decreasing with $G(0) > 0$. A low-pass spectral filter $\phi$ is given in frequency as $\hat{\phi}(n) := G(\lambda_n)$ and its dilation at scale $2^j$ for $j \in \mathbb{Z}$ is $\hat{\phi}_j(n) := G(2^j \lambda_n)$. 

Using the set of low pass filters, $\{\hat{\phi}_j\}_{j \in \mathbb{Z}}$, we define wavelets by
\begin{equation} \label{eq: wavelet definition}
\hat{\psi}_j(n) := [|\hat{\phi}_{j-1}(n)|^2 -  |\hat{\phi}_{j}(n)|^2]^{1/2} = [|G(2^{j-1}n)|^2 -  |G(2^jn)|^2]^{1/2},
\end{equation} which are diffusion wavelets \cite{m2006diffusion}.

Fix $J \in \mathbb{Z}$. Define the operators
\begin{align*}
A_Jf &:= f \ast \phi_J, \\
\Psi_j f &:= f \ast \psi_j, \qquad j \leq J.
\end{align*}
The windowed geometric wavelet transform is given by
\begin{equation}
W_Jf := \{A_J f, \Psi_j f: \quad j  \leq J \}
\end{equation}
and the nonwindowed geometric wavelet transform is given by 
\begin{equation}
Wf := \{\Psi_j f: \quad j \in \mathbb{Z}\}.
\end{equation}

We have the following theorem, which provides a condition for when our wavelet frame is a nonexpansive frame.
\begin{theorem} \label{wavelet frame bounded in l2}
Let $G : [0, \infty) \to \mathbb{R}$ be nonnegative, decreasing, and continuous with $0 < G(0) = C$,  $\lim_{x \to \infty}G(x) = 0$, and $\{\psi_j\}_{j \in \mathbb{Z}}$ is a set of wavelets generated by using the low pass filter $\hat{\phi}(k) = G(\lambda_k)$ in Equation \ref{eq: wavelet definition}. Then we have
\begin{equation}
\sum_{j \in \mathbb{Z}} \|f \ast \psi_j\|_2^2  = C^2 \|f\|_2^2.   
\end{equation}
\end{theorem}
\begin{proof}
For fixed $J > 1$, we telescope to get 
\begin{align*}
\sum_{j = -J}^J |\hat{\psi}_j(k)|^2 &= \sum_{j = - J}^J \left[|G(2^{j-1}\lambda_k)|^2 - |G(2^{j}\lambda_k)|^2\right] \\
&= |G(2^{-J-1}\lambda_k)|^2 - |G(2^{J}\lambda_k)|^2.
\end{align*}
Since $\lim_{J \to \infty} |G(2^{-J-1}\lambda_k)|^2$ and $\lim_{J \to \infty} |G(2^{J}\lambda_k)|^2$ both exist, it follows that 
$$
\sum_{j \in \mathbb{Z}} |\hat{\psi}_j(k)|^2 = \lim_{J \to \infty} |G(2^{-J-1}\lambda_k)|^2 - \lim_{J \to \infty} |G(2^{J}\lambda_k)|^2 = C^2.
$$

We can write
$$\|f \ast \psi_j\|_2^2 = \sum_{k \in \mathbb{N}_0} |\hat{\psi}_j(k)|^2 |\hat{f}(k)|^2.$$ 
Thus, it follows that 
\begin{align*}
\sum_{j \in \mathbb{Z}} \|f \ast \psi_j\|_2^2 &= \sum_{j \in \mathbb{Z}} \sum_{k \in \mathbb{N}_0} |\hat{f}(k)|^2 |\hat{\psi}_j(k)|^2 \\
&= \sum_{k \in \mathbb{N}_0} |\hat{f}(k)|^2 \left(\sum_{j \in \mathbb{Z}}   |\hat{\psi}_j(k)|^2\right) \\
&= C^2\|f\|_2^2.
\end{align*}
\end{proof}


\subsection{The Geometric Scattering Transform}
In an analogous manner to the Euclidean definition of the scattering transform, it is valuable to find a representation that meaningfully encodes high frequency information of a signal $f$. Define the propagator as
\begin{equation}
U[j]f := |\Psi_j f| \qquad \forall j \in \mathbb{Z},
\end{equation} which is convolution of a wavelet and applying a nonlinearity; we can also define the windowed propagator as 
\begin{equation}
U_J[j]f := |\Psi_j f| \qquad \forall j \leq J.
\end{equation} Similar to scattering transforms on Euclidean space, one can apply a cascade of convolutions and modulus operators repeatedly. In particular, for $m \in \mathbb{N}$, let $j_1 , \ldots, j_m \in \mathbb{Z}$. The $m$-layer propagator is defined as 
\begin{equation}
U[j_1, \ldots, j_m] f := U[j_m] \cdots U[j_1]f = | | |f \ast \psi_{j_1}| \ast \psi_{j_2}| \cdots \ast \psi_{j_m}|
\end{equation} 
and the $m$-layer windowed propagator is defined as 
\begin{equation}
U[j_1, \ldots, j_m] f := U[j_m] \cdots U[j_1]f = | | |f \ast \psi_{j_1}| \ast \psi_{j_2}| \cdots \ast \psi_{j_m}|, \qquad j_1, \ldots, j_m \leq J
\end{equation} 
with $U[\emptyset]f = f$ and $U_J[\emptyset]f = f.$ To aggregate low frequency information and get local isometry invariance, one can apply a low pass filter in a manner similar to pooling to each windowed propagator to get windowed scattering coefficients:
\begin{equation*}
S_J[j_1, \ldots, j_m] = A_J U_J[j_1, \ldots, j_m]f = U_J[j_1, \ldots, j_m]f \ast \phi_J,
\end{equation*}
where we defined $S_J[\emptyset]f = f \ast \phi_J$. The windowed geometric scattering transform is given by 
\begin{equation}
S_Jf = \{S_j[j_1, \ldots, j_m]f: \, m \geq 0, \quad j_i \leq J \quad \forall 1 \leq i \leq m\}.
\end{equation} 
The authors of \cite{geometric} were able to prove that this windowed scattering operator was nonexpansive, invariant to isometries up to the scale of the low pass filter, and stable to diffeomporhisms under mild assumptions.

In addition, the authors consider a nonwindowed scattering transform, which removes the low pass filtering. For applications such as manifold classification, one desires full isometry invariance instead of isometry invariance up to scale $2^J$. We see that
\begin{equation}
\lim_{J \to \infty}S[j_1, \ldots, j_m]f = \text{vol}(\mathcal{M})^{-1/2}\|U[j_1, \ldots, j_m]f\|_1.
\end{equation}
As a proxy, one can consider
\begin{equation}
\overline{S}f[j_1, \ldots, j_m] =\|U[j_1, \ldots, j_m]f\|_1,
\end{equation} 
This motivates \cite{geometric} defining the nonwindowed geometric scattering transform as 
\begin{equation}
\overline{S}f = \{\overline{S}[j_1, \ldots, j_m]f: \, m \geq 0, \quad j_i \in \mathbb{Z}, \quad \forall 1 \leq i \leq m\}.
\end{equation}

Instead of considering $\Lb^1(\mathcal{M})$ norms of $m$-layer propagators, we can instead consider $\Lb^q(\mathcal{M})$ norms of $m$-layer propagators for $q \in (1,2]$, which we define as $m$-layer $q$-nonwindowed geometric scattering coefficients (which are also referred to as scattering moments in other works):
\begin{equation}
\overline{S}_q^m[j_1, \ldots, j_m]f = \|U[j_1, \ldots, j_m]f\|_{\Lb^q(\mathcal{M})} \qquad \forall (j_1, \ldots, j_m) \in \mathbb{Z}^m,
\end{equation} which has seen application in quantum chemistry \cite{QMWavelet1, QMWavelet2, QMWavelet3} and for point cloud data \cite{chew2022manifold}. As shorthand notation, we will use the following notation for one layer coefficients:
\begin{equation}
\overline{S}_q[j]f = \|U[j]f\|_{\Lb^q(\mathcal{M})} \qquad \forall j \in \mathbb{Z}.
\end{equation}

To measure stability, invariance, and equivariance, we define the following norm for  $q$-nonwindowed geometric scattering coefficients:
\begin{equation} \label{eqn: scattering norm}
    \| \overline{S}_q^mf  \|_{\ellb^2 (\Z^m)}^q := \left(\sum_{j_m \in \Z} \ldots \sum_{j_1 \in \Z} | \overline{S}_q^m[j_1, \ldots, j_m]f |^2\right)^{q/2},
\end{equation} which follows the definition in \cite{chua2023generalizations}. Since many of the results in \cite{chua2023generalizations} rely on bounds of Littlewood Paley $g$-functions, we will provide extensions of these results to compact manifolds, with some extra restrictions added.

\section{Generalization of the Littlewood Paley $g$-function}
Let $\mathcal{B}$ be a Banach space, suppose that $g: \mathcal{M} \to \mathcal{B}$, and $x \to \|g(x)\|_\mathcal{B}$ is measurable with respect to the measure induced by the Riemannian volume. Define $\Lb^p_\mathcal{B}(\mathcal{M})$ for $1 \leq p < \infty$ to be
\begin{equation*}
    \|g\|_{\Lb^p_\mathcal{B}(\mathcal{M})}^p =  \int_{\mathcal{M}}\|g(x)\|_\mathcal{B}^p \, d\mu (x) \, .
\end{equation*}
Also, for $1 \leq p < \infty$, define
\begin{equation*}
    \|g\|_{\Lb_\mathcal{B}^{p,\infty}(\mathcal{M})} = \sup_{\delta > 0} \delta \cdot \mu(\{x \in \mathcal{M} : \|g(x)\|_{\mathcal{B}} > \delta\})^{1/p} \,,
\end{equation*} It is clear that
\begin{equation*}
    \|g\|_{\Lb_\mathcal{B}^{p,\infty}(\mathcal{M})} \leq \|g\|_{\Lb^p_{\mathcal{B}}(\mathcal{M})} \, .
\end{equation*} 

In particular, let our Banach space be $\ell^2(\mathbb{Z})$, and consider the operator $\vec{T}f: \mathcal{M} \to \ell^2(\mathbb{Z})$ given by 
$$\vec{T}(f)(x) := \{T_j f(x)\}_{j \in \mathbb{Z}} = \left\{\int_{\mathcal{M}}K_{2^{-j/2}}(x,y)f(y)\, d \mu (y) \right\}_{j \in \mathbb{Z}}.$$
with kernel given by $\vec{K} = \{K_{2^{-j/2}}\}_{j \in \mathbb{Z}}$ associated to the wavelets $\{\psi_{j}\}_{j \in \mathbb{Z}}$ generated using a low pass filter $G \in \mathcal{S}(\mathbb{R}^{+})$ satisfying the conditions satisfying the conditions of Theorem \ref{wavelet frame bounded in l2} for Equation \ref{eq: wavelet definition}. This can be seen as an extension of the Littlewood Paley $g$-function for a compact Riemannian manifold since
$$\|\vec{T}(f)(x)\|_{\ell^2(\mathbb{Z})} = \left(\sum_{j \in \mathbb{Z}} |f \ast \psi_j(x)|^2\right)^{1/2}.$$

From Theorem \ref{wavelet frame bounded in l2}, 
$$\|\vec{T}f\|_{\Lb^2_{\ell^2(\mathbb{Z})}(\mathcal{M})} \leq \|f \|_{\Lb^2(\mathcal{M})}.$$
Our goal is to extend this operator and prove that for all $q \in (1,2)$, there exists $C_q$ such that 
$$\|\vec{T}f\|_{\Lb^q_{\ell^2(\mathbb{Z})}(\mathcal{M})} \leq C_q\|f \|_{\Lb^q(\mathcal{M})}.$$
Before providing any proofs, we will state preliminary lemmas that will be vital to our approach. The first few lemmas concern the kernel of our convolution operator.

\begin{lemma}[\cite{Geller_Pesenson_2014}]\label{kernel decya condition} Suppose that $F \in \mathcal{S}(\mathbb{R}^{+})$, the space of Schwartz functions restricted to $\mathbb{R}^{+}$ and $F(0) = 0$. For the kernel
$$K_t(x,y) = \sum_{n \in \mathbb{N}} F(t^2 \lambda_n) e_n(x) \overline{e_n(y)},$$
the following pointwise bound holds for some $C_n > 0$ and for all $t > 0$ and all $x,y \in \mathcal{M}$:
$$|K_t(x,y)| \leq \frac{C_n t^{-n}}{\left(1 + \tfrac{r(x,y)}{t}\right)^{n+1}},$$
We remind the reader that $n$ is defined as the dimension of $\mathcal{M}$.
\end{lemma}


Let $\mathcal{U}$ be an open cover for $\mathcal{M}$. A number $\gamma > 0$ is called a Lebesgue number for $\mathcal{U}$ if and only if for all $x \in \mathcal{M}$, there exists $\mathcal{U}_x \in \mathcal{U}$ such that $B(x, \gamma) \subset \mathcal{U}_x$, where we define
\begin{equation}
B(x,r_0) := \{ y \in X \, : \, r(x,y) < r_0\}.
\end{equation}

\begin{lemma}[\cite{geller2009continuous}] \label{manifold regularity}Cover $\mathcal{M}$ with a finite collection of open sets $P_i$ with $1 \leq i \leq I$ such that the following properties hold for each index $i$:
\begin{enumerate}
\item there exists a chart $(V_i, \varphi_i)$ with $\overline{P}_i \subset V_i$
\item $\varphi_i(P_i)$ is a ball in $\mathbb{R}^n$.
\end{enumerate} 

Choose $\delta >0$ so that $3 \delta$ is a Lebesgue number for the covering $\{P_i\}$. Then there exist $c_1, c_2 > 0$ such that for any $x \in \mathcal{M}$ and any $B(x, 3 \delta) \subset P_i$, the following statements hold in the coordinate system on $P_i$ obtained from $\varphi_i$:
\begin{enumerate}
\item For all $y, z \in P_i$, we have $r(y,z) \leq c_2|y-z|$.
\item For all $y, z \in B(x,\delta)$, we have $r(y,z) \geq c_1|y-z|$.
\end{enumerate}
\end{lemma}

\begin{proof}
This proof follows the idea for \cite[Proposition 3.1]{geller2009continuous} yet with explicitly identified constants $c_1,c_2$ to demonstrate their continuous dependence on the metric $g$.
Using the coordinates given by $\varphi_i$, we identify $g(x)$ with an $n\times n$ smooth matrix for each $x\in\overline{P}_i$. We write $|\cdot|_g$ for the norm induced by $g$, and $|\cdot|$ for the Euclidean norm in the $\varphi_i$-coordinates.

For each $x\in\overline{P}_i$ and a tangent vector $v_x\in T_x\mathcal{M}$, we have
$$
\underline{\Lambda}(x) |v_x|^2 \leq  |v_x|^2_{g(x)} = \sum^n_{\alpha,\beta=1} v_x^\alpha g_{\alpha\beta}(x) v_x^\beta \leq \overline{\Lambda}(x) |v_x|^2.
$$
Here $\underline{\Lambda}(x), \overline{\Lambda}(x)$ denote the minimum and maximum eigenvalues of the matrix $g(x)$ respectively. We conclude
$$
\left( \min_{x\in\overline{P}_i} \underline{\Lambda}(x) \right) |v_x|^2 \leq  |v_x|^2_{g(x)} \leq \left( \max_{x\in\overline{P}_i} \overline{\Lambda}(x) \right) |v_x|^2
$$
uniformly for $(x,v)\in T\overline{P}_i$. As both $\underline{\Lambda}(x)$ and $\overline{\Lambda}(x)$ are positive continuous functions of $x$ on the compact subset $\overline{P}_i$, their minimum and maximum are strictly positive.

For the first statement, take $y,z\in P_i$ and define $\gamma:[0,1]\rightarrow P_i$, $\gamma(t):=tz+(1-t)y$ where the coordinates of $y,z$ are given by $\varphi_i$. Then $\gamma$ is a smooth curve connecting $y,z$ and $|\gamma'(t)| = |y-z|$ for all $t$. By the definition of the Riemannian distance $r(y,z)$, we have
$$
    r(y,z) \leq \text{ length of } \gamma = \int^1_0 |\gamma'(t)|_{g(\gamma(t))} \, dt 
    \leq \left( \max_{x\in\overline{P}_i} \overline{\Lambda}(x) \right)^{\frac{1}{2}} \int^1_0 |\gamma'(t)| \, dt 
    = \left( \max_{x\in\overline{P}_i} \overline{\Lambda}(x) \right)^{\frac{1}{2}} |y-z|.
$$
One choice for $c_2$ is $c_2:= \max_{i=1,\dots,I}\left( \max_{x\in\overline{P}_i} \overline{\Lambda}(x) \right)^{\frac{1}{2}}$.

For the second statement, take $y, z \in B(x,\delta)$ and let $\gamma_k: [0,1] \rightarrow \mathcal{M}$ be a sequence of piecewise $C^1$ curves connecting $y,z$ such that their lengths $\ell(\gamma_k)\rightarrow r(y,z)$ as $k\rightarrow\infty$. For large $k$, we have 
$$
r(\gamma_k(t),x) \leq r(\gamma_k(t),y) + r(y,x) \leq r(z,y) + \delta \leq r(z,x) + r(x,y) + \delta \leq 3\delta
$$
for all $t\in [0,1]$, hence $\gamma_k\subset P_i$. Therefore, we have for large $k$ that
$$
\ell(\gamma_k) = \int^1_0 |\gamma'_k(t)|_g \, dt
\geq \left( \min_{x\in\overline{P}_i} \underline{\Lambda}(x) \right)^{\frac{1}{2}} \int^1_0 |\gamma'_k(t)| \,dt
\geq  \left( \min_{x\in\overline{P}_i} \underline{\Lambda}(x) \right)^{\frac{1}{2}} \left| \int^1_0 \gamma'_k(t) \,dt 
\right|
\geq \left( \min_{x\in\overline{P}_i} \underline{\Lambda}(x) \right)^{\frac{1}{2}} |y-z|.
$$
Let $k\rightarrow\infty$ proves the statement, and one choice for $c_1$ is $c_1:=\min_{i=1,\dots,I} \left( \min_{x\in\overline{P}_i} \underline{\Lambda}(x) \right)^{\frac{1}{2}}$.

\end{proof}

For the rest of this paper, we fix collections $\{P_i\}$, $\{V_i\}$, $\{\varphi_i\}$, and constants $\delta,c_1,c_2$ from the previous Lemma.

\begin{lemma} \label{kernel regularity bound} Suppose that $r(y,z) < \min\left\{\tfrac{1}{2}r(x,y), \delta\right\}$ so that $y$ and $z$ lie on the same ball of the covering. Assume that there exist $c_1$ and $c_2$ from Lemma \ref{manifold regularity} with $c_1c_2 < 2.$ Then there exists a constant $C_\delta$ such that
$$|K_t(x,y)-K_t(x,z)| \leq C_\delta \frac{r(y,z) t^{-n-1}}{\left(1+ \frac{r(x,y)}{t}\right)^{n+1}}.$$
\end{lemma}
\begin{proof} Using the proof of Theorem 5.5 in \cite{geller2009continuous}, for each $x \in \mathcal{M}$, there exists a point $w_x$ on the segment connecting $y$ to $z$ such that
$$|K_t(x,y)-K_t(x,z)| \leq C_1 \frac{r(y,z) t^{-n-1}}{\left(1+ \frac{r(x,w_x)}{t}\right)^{n+1}}.$$
Now notice that triangle inequality implies that
$$r(x,y) \leq r(x,w_x) + r(y,w_x).$$
By Lemma \ref{manifold regularity}, since $w_x$ lies on the line segment between $y$ and $z$, we see that 
 $$r(y,w_x) \leq c_2|y-w_x| \leq c_2|y-z| \leq c_1c_2r(y,z).$$
It follows that
$$r(x,y) - c_1c_2r(y,z)\leq r(x,w_x).$$
Since $c_1c_2 < 2$ and $r(x,y) \geq 2 r(y,z)$, we see that
$$r(x,y) - \frac{c_1c_2}{2} r(x,y) \leq r(x,w_x)$$
so that $1 - \tfrac{c_1c_2}{2} > 0$. This leads to
$r(x,y) \leq C_r r(x,w_x)$ for some constant $C_r$ independent of $x$. Finally, we have 
$$|K_t(x,y)-K_t(x,z)| \leq C_\delta \frac{r(y,z) t^{-n-1}}{\left(1+ \frac{r(x,y)}{t}\right)^{n+1}}.$$
\end{proof}

Now we provide the necessary tools from classical harmonic analysis for extension. The idea is similar to the proof in the Euclidean case; we wish to prove a weak-type $(1,1)$ bound and extend by interpolation. Let $X$ be a set, $\beta$ a quasimetric, and $\mu$ a measure on $X$ such that 
$0 < \mu(B(x,r)) < \infty$ for all $x \in X$ and $r > 0$. We say that a space $(X, \beta, \mu)$ is of homogeneous type, though often with the quasi-metric and measure omitted when implied, if for all $x \in X$ and $r > 0$ there exists a constant $C_D$ such that
\begin{equation}
\mu(B(x,2r)) \leq C_D \mu(B(x,r)),
\end{equation}
where $B(x,r)$ is a ball of radius $r$ centered at $x$ for $(X, \beta).$ The property above is also known as the doubling property. It is well known that a $C^\infty$ compact Riemannian manifold using the standard Riemannian metric and volume is of homogeneous type. 

The first result we will need is a Calderon-Zygmund decomposition:

\begin{theorem}[\cite{Coifman1971AnalyseHN}] Suppose that $X$ is a space of homogeneous type. Suppose that $f \in \Lb^1(X)$ and choose $\alpha > 0$ such that $\alpha^{-1}\|f\|_1 < \mu(X)$. Then we can decompose $f := g+b$ such that
\begin{align*}
\|g\|_{\Lb^2(\mathcal{M})}^2 &\leq C_1 \alpha \|f\|_{\Lb^1(\mathcal{M})}, \\
b &= \sum_{i} b_i,
\end{align*} where $C_1>0$ is a constant, $\text{supp}(b_i) \subset B(x_i, r_i)$ for some countable collection of balls $\{B(x_i, r_i)\}$, and each $b_i$ satisfies 
\begin{align*}
\int_{X} b_i(x) d\mu(x) &= 0, \\
\|b_i\|_1 &\leq C \alpha \mu(B(x_i, r_i)), \\
\sum_{i} \mu(B(x_i, r_i)) &\leq C \alpha^{-1}\|f\|_{\Lb^1(\mathcal{M})}.
\end{align*}
\end{theorem}

\begin{lemma} Suppose that we choose wavelets $\{\psi_{j}\}_{j \in \mathbb{Z}}$ generated using $G \in \mathcal{S}(\mathbb{R}^{+})$ in Equation \ref{eq: wavelet definition} that satisfy the conditions of Theorem \ref{wavelet frame bounded in l2} and $c_1 c_2 < 2$ in Lemma \ref{manifold regularity}. The following weak $(1,1)$ bound holds for some $A > 0$:
$$\|\vec{T}f\|_{\Lb^{1,\infty}_{\ell^2(\mathbb{Z})}(\mathcal{M})} \leq  A\|f\|_{\Lb^1(\mathcal{M})}.$$
\end{lemma}

\begin{proof}
First, for any $\alpha$ such that $\alpha^{-1} \|f\|_{\Lb^1(\mathcal{M})} > \mu(\mathcal{M})$, we see that
\begin{align*}
\mu(\{x \in \mathcal{M} : \|\vec{T}f(x)\|_{\mathcal{B}} > \alpha \}) \leq \mu(\mathcal{M}) \leq \alpha^{-1}\|f\|_{\Lb^1(\mathcal{M})}.
\end{align*}
Now, we consider the case where $\alpha^{-1} \|f\|_{\Lb^1(\mathcal{M})} \leq \mu(\mathcal{M})$. We use our Caulderon-Zygmund decomposition and write $f = g+b$. It follows that 
\begin{align*}
\mu(\{x \in \mathcal{M} : \|\vec{T}f(x)\|_{\mathcal{B}} > \alpha\}) &\leq \mu(\{x \in \mathcal{M} : \|\vec{T}g(x)\|_{\mathcal{B}} > \alpha/2\}) + \mu(\{x \in \mathcal{M} : \|\vec{T}b(x)\|_{\mathcal{B}} > \alpha/2\})\\
& := I_1 + I_2.
\end{align*}
For $I_1$, we apply Chebyshev inequality, $\Lb^{2}_{\ell^2(\mathbb{Z})}(\mathcal{M})$ boundedness of $\vec{T}$, and our assumption on $g$ to find that  
\begin{align*}
\mu(\{x \in \mathcal{M} : \|\vec{T}g(x)\|_{\mathcal{B}} > \alpha/2\}) &\leq \frac{4}{\alpha^2} \|\vec{T}g\|_{\Lb^{2}_{\ell^2(\mathbb{Z})}(\mathcal{M})}^2 \\
& \leq \frac{4}{\alpha^2} \|\vec{T}\|_{\Lb^{2}(\mathcal{M}) \to \Lb^{2}_{\ell^2(\mathbb{Z})}(\mathcal{M})}^2 \|g\|_{\Lb^2(\mathcal{M})}^2 \\
& \leq \frac{4 C_1}{\alpha} \|\vec{T}\|_{\Lb^{2}(\mathcal{M}) \to \Lb^{2}_{\ell^2(\mathbb{Z})}(\mathcal{M})}^2\|f\|_{\Lb^1(\mathcal{M})}.
\end{align*}
For $I_2$, let $B = \bigcup_i B(x_i, 2r_i)$. Then it follows that 
\begin{align*}
I_2 &\leq \mu(B) + \mu(\{x \in B^{c} : \|\vec{T}b(x)\|_{\mathcal{B}} > \alpha/2\}\\
&\leq \frac{C_D}{\alpha} \|f\|_{\Lb^1(\mathcal{M})} + \frac{2}{\alpha} \| \vec{T}b \|_{\Lb^{1}_{\ell^2(\mathbb{Z})}(B^c)}  \\
&\leq \frac{C_D}{\alpha} \|f\|_{\Lb^1(\mathcal{M})} + \frac{2}{\alpha} \sum_{i} \| \vec{T}b_i \|_{\Lb^{1}_{\ell^2(\mathbb{Z})}(B(x_i, 2r_i)^c)},
\end{align*} 
where the constant $C_D$ comes from the fact that our measure has the doubling property.

To estimate $\| \vec{T}b_i \|_{\Lb^{1}_{\ell^2(\mathbb{Z})}(B(x_i, 2r_i)^c)}$, we notice that
\begin{align*}
\| \vec{T}b_i \|_{\Lb^{1}_{\ell^2(\mathbb{Z})}(B(x_i, 2r_i)^c)} &= \int_{B(x_i, 2r_i)^c} \| \vec{T}b_i(x) \|_{\ell^2(\mathbb{Z})} \, d \mu(x)\\
& = \int_{B(x_i, 2r_i)^c} \left\|\int_{B(x_i, r_i)} \vec{K}(x,y ) b_i(y)\, dy \right\|_{\ell^2(\mathbb{Z})} \, d\mu(x) \\ 
&  = \int_{B(x_i, 2r_i)^c} \left\|\int_{B(x_i, r_i)} (\vec{K}(x,y )-\vec{K}(x,x_i))b_i(y)\, d \mu(y) \right\|_{\ell^2(\mathbb{Z})} \, d\mu(x)\\
& = \int_{B(x_i, r_i)} \left(\int_{B(x_i, 2r_i)^c} \left(\sum_{j \in \mathbb{Z}}\left| {K}_{2^{-j/2}}(x,y )-{K}_{2^{-j/2}}(x,x_i)\right|^2 |b_i(y)|^2\right)^{1/2}  \, d\mu(x)\right) \, d\mu(y) \\ 
& \leq \int_{B(x_i, r_i)} |b_i(y)| \int_{B(x_i, 2r_i)^c}\sum_{j \in \mathbb{Z}}\left| {K}_{2^{-j/2}}(x,y )-{K}_{2^{-j/2}}(x,x_i)\right|  \, d\mu(x)  \, d\mu(y) \\
&= \int_{B(x_i, r_i)} |b_i(y)| \left(\sum_{j \in \mathbb{Z}}\int_{B(x_i, 2r_i)^c}\left| {K}_{2^{-j/2}}(x,y )-{K}_{2^{-j/2}}(x,x_i)\right|   \, d\mu(x)\right)  \, d\mu(y),
\end{align*} 
where $\vec{K}$ is the kernel defined on page 7. Now we consider the term inside the parentheses. We will break this argument into cases. First, consider if $2r_i \geq \delta$. We see that 
\begin{align*}
&\sum_{j \in \mathbb{Z}} \int_{B(x_i, 2r_i)^c}  |K_{2^{-j/2}}(x,y )-K_{2^{-j/2}}(x,x_i)| \, d\mu(x) \\
&\leq C \sum_{j \in \mathbb{Z}} \int_{B(x_i, 2r_i)^c} \frac{2^{nj/2}}{\left(1 + 2^{j/2} r(x,x_i)\right)^{n+1}}  + \frac{2^{nj/2}}{\left(1 + 2^{j/2} r(x,y)\right)^{n+1}} \, d\mu(x).
\end{align*}

Now, since $x_i$ is the center of $B(x_i, r_i)$, if $x \in B(x_i, 2r_i)^c$, then $r(x,x_i) \geq 2r_i \geq \delta$. Similarly, since $y \in B(x_i, r_i)$, it follows that $r(y, x_i) < r_i$ and we have $2r(y,x_i) \leq r(x,x_i)$. Apply triangle inequality to get
$$r(x,x_i) \leq r(x,y) + r(y,x_i) \leq r(x,y) + r_i \leq r(x,y) + \frac{1}{2}r(x,x_i),$$
which means that $r(x,x_i) \leq 2r(x,y).$

Going back to the integral, there exists $C_1$ such that
\begin{align*}
&\sum_{j \in \mathbb{Z}} \int_{B(x_i, 2r_i)^c} \frac{2^{nj/2}}{\left(1 + 2^{j/2} r(x,y)\right)^{n+1}}  + \frac{2^{nj/2}}{\left(1 + 2^{j/2} r(x,x_i)\right)^{n+1}} \, d\mu(x) \\
&\leq C_1 \sum_{j \in \mathbb{Z}} \int_{r(x,x_i) \geq 2r(y, x_i)} \frac{2^{nj}}{\left(1 + 2^j r(x,x_i)\right)^{n+1}} \, d\mu(x)\\
&= C_1 \sum_{j \geq 0} \int_{r(x,x_i) \geq 2r(y, x_i)} \frac{2^{nj/2}}{\left(1 + 2^{j/2} r(x,x_i)\right)^{n+1}} \, d\mu(x) + C_1 \sum_{j <0 } \int_{r(x,x_i) \geq 2r(y, x_i)} \frac{2^{nj/2}}{\left(1 + 2^{j/2} r(x,x_i)\right)^{n+1}} \, d\mu(x) \\
&:= J_1 + J_2.
\end{align*}
For $J_1$, since $r(x,x_i) > \delta$, 
$$ \sum_{j \geq 0}\int_{r(x,x_i) \geq 2r(y, x_i)} \frac{2^{nj/2}}{\left(1 + 2^{j/2} r(x,x_i)\right)^{n+1}} \, d\mu(x) \leq   \sum_{j \geq 0} 2^{-j/2} \delta^{n+1} < \infty.$$
For $J_2$, it is routine to see that  
$$ \sum_{j < 0}\int_{r(x,x_i) \geq 2r(y, x_i)} \frac{2^{nj/2}}{\left(1 + 2^{j/2} r(x,x_i)\right)^{n+1}} \, d\mu(x) \leq \sum_{j < 0} 2^{nj/2} \mu(\mathcal{M}) < \infty.$$

Now we consider the case where $2r_i < \delta$. In this case, we see that $r(y, x_i) < r_i < \delta$, and we still have $2r(y,x_i) < r(x,x_i)$. Thus, the bound 
$$|K_{2^{-j/2}}(x,y )-K_{2^{-j/2}}(x,x_i)| \leq C\frac{2^{nj/2}}{\left(1 + 2^{j/2} r(x,x_i)\right)^{n+1}}$$
still applies. We can also apply Lemma \ref{kernel regularity bound} to get 
$$|K_{2^{-j/2}}(x,y)-K_{2^{-j/2}}(x,x_i)| \leq C_\delta \frac{r(y,x_i) 2^{(n+1)j/2}}{\left(1+ 2^{j/2}r(x,x_i)\right)^{n+1}}.$$

Taking the geometric mean, for any $s \in [0,1]$, we have
$$|K_{2^{-j/2}}(x,y)-K_{2^{-j/2}}(x,x_i)| \leq C_2  \frac{2^{nj/2} (2^{j/2} r(y,x_i))^s}{\left(1+ 2^{j/2}r(x,x_i)\right)^{n+1}}$$
for some constant $C_2$. It now follows that for $C_3 = \max\{C_\delta, C_2\}$, we have
\begin{align*}
&\sum_{j \in \mathbb{Z}}|K_{2^{-j/2}}(x,y)-K_{2^{-j/2}}(x,x_i)| \\
& \leq \sum_{2^{j/2} < \frac{2}{r(x,x_i)}}|K_{2^{-j/2}}(x,y)-K_{2^{-j/2}}(x,x_i)|  + \sum_{2^{j/2} \geq \frac{2}{r(x,x_i)}}|K_{2^{-j/2}}(x,y)-K_{2^{-j/2}}(x,x_i)| \\
& \leq C_\delta \sum_{2^{j/2} < \frac{2}{r(x,x_i)}}\frac{r(y,x_i) 2^{(n+1)j/2}}{\left(1+ 2^{j/2}r(x,x_i)\right)^{n+1}} + C_2\sum_{2^{j/2} \geq \frac{2}{r(x,x_i)}} \frac{2^{nj/2} (2^{j/2} r(y,x_i))^{1/2}}{\left(1+ 2^{j/2}r(x,x_i)\right)^{n+1}} \\
&\leq C_3 \left(r(x_i,y) \sum_{2^{j/2} < \frac{2}{r(x,x_i)}} 2^{(n+1)j/2}  + r(x_i,y)^{1/2}\sum_{2^{j/2} \geq \frac{2}{r(x,x_i)}} 2^{(n+1/2)j/2}(2^{j/2}r(x,x_i))^{-n-1}\right)\\
& \leq C_3(r(x_i,y) r(x,x_i)^{-n-1} + r(x_i,y)^{1/2} r(x,x_i)^{-n-1/2}).
\end{align*} 
Integrating over $2r(x_i,y) \leq r(x,x_i)$ yields a constant independent of $r_i$. It now follows that 
$$\| \vec{T}b_i \|_{\Lb^{1}_{\ell^2(\mathbb{Z})}(B(x_i, r_i)^c)} \leq C_4 \|b_i\|_{\Lb^1(\mathcal{M})}$$
for some constant $C_4$. Using the Caulderon-Zygmund decomposition,
$$I_2 \leq \left(\frac{C_D}{\alpha} + \frac{2 C_4}{\alpha}\right) \|f\|_{\Lb^1(\mathcal{M})}.$$
\end{proof}

Recall the following result, which is a vector-valued version of Marcinkiewicz Interpolation: 
\begin{lemma} [\cite{vectorvaluedinequalities}] 
Let $\mathcal{A}_1, \mathcal{A}_2$ be Banach spaces,  $T:\mathcal{A}_1 \to \mathcal{A}_2$ be quasilinear on $\Lb^{p_0}_{\mathcal{A}_1}(X)$ and $\Lb^{p_1}_{\mathcal{A}_1}(X)$ with $0 < p_0 < p_1$. If $T$ satisfies
$$\|Tf\|_{\Lb^{p_i,\infty}_{\mathcal{A}_2}(X)} \leq M_i\|f\|_{\Lb_{\mathcal{A}_1}^{p_i}(X)}$$ for $i = 0, 1$, then
$$\|Tf\|_{\Lb^p_{\mathcal{A}_2(X)}} \leq N_p \|f\|_{\Lb_{\mathcal{A}_1}^{p}(X)} \qquad \forall p \in (p_0, p_1),$$
where $N_p$ is dependent on $p$.
\end{lemma}

The following corollary is a direct result of interpolation now:
\begin{corollary} \label{operator extension theorem} Suppose that we choose wavelets $\{\psi_{j}\}_{j \in \mathbb{Z}}$ generated by using $G \in \mathcal{S}(\mathbb{R}^{+})$ in Equation \ref{eq: wavelet definition}, $G$ satisfies the conditions of Theorem \ref{wavelet frame bounded in l2}, and $c_1 c_2 < 2$. We have
$$\|\vec{T}f\|_{\Lb^q_{\ell^2(\mathbb{Z})}(\mathcal{M})}^q \leq C_q \|f \|_{\Lb^q(\mathcal{M})}^q$$
for some constant $C_q > 0$, where $q \in (1,2)$.
\end{corollary}

By duality, the result of Corollary \ref{operator extension theorem} actually holds for $q \in(1, \infty)$. However, for generalizing the nonwindowed geometric scattering transform, we only need results for $q \in (1,2)$ since $\Lb^2(\mathcal{M}) \subset \Lb^q(\mathcal{M})$. For $q > 2$, since our manifold is compact, we have $\Lb^q(\mathcal{M}) \subset \Lb^2(\mathcal{M})$, so previous results in \cite{geometric} are applicable, and further theoretical analysis is not as significant.

Additionally, although the result of Corollary \ref{operator extension theorem} seems restrictive because one needs $c_1c_2 < 2$, the result applies for a variety of different manifolds. If one finds a metric where the condition above holds, a class of metrics can be found by perturbing the metric. 
This is because the choice of the constants $c_1,c_2$ in the proof of Lemma \ref{manifold regularity} depend continuously on the metric $g$. Thus if $c_1c_2 < 2$ for $g$, the same strict inequality holds for all metrics that are sufficiently close to $g$. 
We provide a simple example below where the conditions of Lemma \ref{manifold regularity} hold. The result of the example below can also be extended to $n$-torii without much difficulty.

\begin{example}
Take $R=1$ and use the charts:
$$
\begin{array}{ll}
    V_1 := \{(x_1,x_2): x_1^2+x_2^2=1, \; x_1>0\}, & \qquad 
    \varphi_1: V_1\rightarrow\mathbb{R}, \; (x_1,x_2)\mapsto x_2 \\
    V_2 := \{(x_1,x_2): x_1^2+x_2^2=1, \; x_2>0\}, & \qquad 
    \varphi_2: V_2\rightarrow\mathbb{R}, \; (x_1,x_2)\mapsto x_1 \\
    V_3 := \{(x_1,x_2): x_1^2+x_2^2=1, \; x_1<0\}, & \qquad 
    \varphi_3: V_3\rightarrow\mathbb{R}, \; (x_1,x_2)\mapsto x_2 \\
    V_4 := \{(x_1,x_2): x_1^2+x_2^2=1, \; x_2<0\}, & \qquad 
    \varphi_4: V_4\rightarrow\mathbb{R}, \; (x_1,x_2)\mapsto x_1.
\end{array}
$$
These are clearly diffeomorphisms. Choose
$$P = \{(-\tfrac{\pi}{3}+\omega, \tfrac{\pi}{3}-\omega), (\tfrac{\pi}{6}+\omega, \tfrac{5 \pi}{6}-\omega), (\tfrac{2\pi}{3}+\omega, \tfrac{4 \pi}{3}-\omega), (\tfrac{7 \pi}{6}+\omega, \tfrac{11 \pi}{6}-\omega)\} := \{P_1, P_2, P_3, P_4\}.$$
Here $\omega\in (0,\frac{\pi}{12})$ is a small angle.
The covers $P_i$ clearly satisfy the first two conditions laid out in Lemma \ref{manifold regularity}. Equip $\mathbb{S}^1$ with the standard metric induced by the inclusion $\mathbb{S}^1 \hookrightarrow \mathbb{R}^2$. We can verify the desired estimates in Lemma \ref{manifold regularity} as follows: It is clear that any arc of $\mathbb{S}^1$ with length less than $\tfrac{\pi}{6}$ is contained in one of $P_1, \dots, P_4$. If we choose $\delta\in (0,\tfrac{\pi}{36})$, then for any $x\in\mathbb{S}^1$, $B(x,3\delta)\subset P_i$ for some $i$. Suppose
\begin{align*}
y = (y_1,y_2) = (\cos\theta,\sin\theta), \\
z = (z_1,z_2) = (\cos\tilde\theta,\sin\tilde\theta), \\
\theta,\tilde\theta\in (-\tfrac{\pi}{3}+\omega, \tfrac{\pi}{3}-\omega), \\
y_2,z_2\in (\sin(-\tfrac{\pi}{3}+\omega), \sin(\tfrac{\pi}{3}-\omega)).
\end{align*}
The $\varphi_1$-coordinates of $y,z$ are $y_2, z_2$, respectively. Hence,
\begin{align*}
|y-z| & = |y_2-z_2| = |\sin\theta - \sin\tilde\theta| \\
r(y,z) & = |\theta-\tilde\theta|.
\end{align*}
By the mean value theorem:
$$
|y-z| = |\sin\theta - \sin\tilde\theta| = |\cos\xi|_{\xi\in (-\tfrac{\pi}{3}+\omega, \tfrac{\pi}{3}-\omega)} |\theta-\tilde\theta| \leq |\theta-\tilde\theta| = r(y,z).
$$
and
\begin{align*}
r(y,z) & = |\theta-\tilde\theta| \\
& = |\arcsin y_2 - \arcsin z_2| \\
&= \left|\frac{1}{\sqrt{1-\eta^2}}\right|_{\eta\in (\sin(-\tfrac{\pi}{3}+\omega), \sin(\tfrac{\pi}{3}-\omega))} |y_2-z_2| \\
 & \leq \frac{1}{\sqrt{1-\sin^2\left( \frac{\pi}{3}-\omega \right)}} |y_2-z_2| \\
 &= \frac{1}{\sqrt{1-\sin^2\left( \frac{\pi}{3}-\omega \right)}} |y-z|.
\end{align*}
This suggests the choice $c_1=1$ and $c_2 = \frac{1}{\sqrt{1-\sin^2\left( \frac{\pi}{3}-\omega \right)}}$. We have
$$
c_1 c_2 = \frac{1}{\sqrt{1-\sin^2\left( \frac{\pi}{3}-\omega \right)}} < \frac{1}{\sqrt{1-\sin^2\left( \frac{\pi}{3} \right)}} = 2.
$$

The analysis for $y,z\in P_3$ is similar, only with the difference that $\theta,\tilde\theta\in (\tfrac{2\pi}{3}+\omega, \tfrac{4\pi}{3}-\omega)$ and we have $y_2,z_2\in (\sin(\tfrac{4\pi}{3}-\omega), \sin(\tfrac{2\pi}{3}+\omega)) = (\sin(-\tfrac{\pi}{3}+\omega), \sin(\tfrac{\pi}{3}-\omega))$.

Next, consider $y,z\in P_2$. Suppose
\begin{align*}
y = (y_1,y_2) = (\cos\theta,\sin\theta), \\
z = (z_1,z_2) = (\cos\tilde\theta,\sin\tilde\theta), \\
\theta,\tilde\theta\in (\tfrac{\pi}{6}+\omega, \tfrac{5\pi}{6}-\omega), \\
y_1,z_1\in (\cos(\tfrac{5\pi}{6}-\omega), \cos(\tfrac{\pi}{6}+\omega)).
\end{align*}
The $\varphi_2$-coordinates of $y,z$ are $y_1, z_1$, respectively. Hence,
\begin{align*}
|y-z| & = |y_1-z_1| = |\cos\theta - \cos\tilde\theta| \\
r(y,z) & = |\theta-\tilde\theta|.
\end{align*}
By the mean value theorem:
$$
|y-z| = |\cos\theta - \cos\tilde\theta| = |\sin\xi|_{\xi\in (\tfrac{\pi}{6}+\omega, \tfrac{5\pi}{6}-\omega)} |\theta-\tilde\theta| \leq |\theta-\tilde\theta| = r(y,z).
$$
and
\begin{align*}
r(y,z) & = |\theta-\tilde\theta| \\
&= |\arccos y_1 - \arccos z_1| \\
&= \left|\frac{1}{\sqrt{1-\eta^2}}\right|_{\eta\in (\cos(\tfrac{5\pi}{6}-\omega), \cos(\tfrac{\pi}{6}+\omega))} |y_1-z_1| \\
 & \leq  \frac{1}{\sqrt{1-\cos^2(\tfrac{\pi}{6}+\omega)}} |y_1-z_1|\\
 &= \frac{1}{\sqrt{1-\cos^2(\tfrac{\pi}{6}+\omega)}} |y-z|.  
\end{align*}
This suggests the choice $c_1=1$ and $c_2 = \frac{1}{\sqrt{1-\cos^2(\tfrac{\pi}{6}+\omega)}}$. We have
$$
c_1 c_2 = \frac{1}{\sqrt{1-\cos^2(\tfrac{\pi}{6}+\omega)}} < \frac{1}{\sqrt{1-\cos^2\left( \frac{\pi}{6} \right)}} = 2.
$$
Note that the choice agrees with the case $y,z\in P_1$. The analysis for $y,z\in P_4$ is similar as well, which proves the desired claim.
\end{example}

\section{Generalizing Nonwindowed Geometric Scattering}
Now that we have developed the machinery necessary for the rest of the paper, we prove the $q$-nonwindowed scattering transforms are bounded operators with respect to \eqref{eqn: scattering norm} and outline basic properties of the representation.

\subsection{The $2$-nonwindowed Geometric Scattering Norm}
We start by proving that 2-nonwindowed scattering transforms are bounded operators with respect to \eqref{eqn: scattering norm}.

\begin{theorem} \label{thm: q = 2 nonexpansive}
Suppose that $G$ satisfies the conditions of Theorem \ref{wavelet frame bounded in l2} and $\{\psi_j\}_{j \in \mathbb{Z}}$ be a set of spectral filters generated by using $G$ in Equation \ref{eq: wavelet definition}. Then we have
$$\|\overline{S}_2^m f- \overline{S}_2^m g\|^2_{\ell^2(\mathbb{Z}^m)} \leq  C^{2m}\|f - g\|_{\Lb^2(\mathcal{M})}^2$$
for all $f, g \in \Lb^2(\mathcal{M})$.
\end{theorem}
\begin{proof}
In the case of $m  = 1$, we see that 
\begin{align*}
\sum_{j \in \mathbb{Z}} |\overline{S}_2f[j] - \overline{S}_2g[j]|^2 &= \sum_{j \in \mathbb{Z}} |\|f \ast \psi_j\|_2 - |g \ast \psi_j\|_2|^2 \\
&\leq \sum_{j \in \mathbb{Z}} \|f \ast \psi_j-g \ast \psi_j\|_2^2 \\
&= \sum_{j \in \mathbb{Z}} \|(f-g) \ast \psi_j\|_2^2 \\
&\leq C^2 \|f-g\|_{\Lb^2(\mathcal{M})}^2.
\end{align*}
For the $m+1$ case, we proceed recursively and see that 
\begin{align*}
&\sum_{(j_1, \ldots, j_{m+1}) \in \mathbb{Z}^{m+1}} \left|\overline{S}_2^{m+1}f[j_1, \ldots, j_{m+1}]- \overline{S}_2^{m+1}g[j_1, \ldots, j_{m+1}]\right|^2 \\
&= \sum_{(j_1, \ldots, j_{m+1}) \in \mathbb{Z}^{m+1}} \left|\|U[j_1, \ldots, j_m]f \ast \psi_{j+1} \|_2 - \|U[j_1, \ldots, j_m]g \ast \psi_{j+1} \|_2\right|^2 \\
&\leq \sum_{(j_1, \ldots, j_{m+1}) \in \mathbb{Z}^{m+1}} \|(U[j_1, \ldots, j_m]f - U[j_1, \ldots, j_m]g) \ast \psi_{j+1} \|_2^2 \\
&\leq C^2\sum_{(j_1, \ldots, j_m) \in \mathbb{Z}^{m}} \|U[j_1, \ldots, j_m]f - U[j_1, \ldots, j_m]g \|_2^2 \\
&= C^2 \sum_{(j_1, \ldots, j_{m}) \in \mathbb{Z}^{m}} \||U[j_1, \ldots, j_{m-1}]f  \ast \psi_{j_m}| - |U[j_1, \ldots, j_{m-1}]g  \ast \psi_{j_m}|\|_2^2 \\
&\leq C^2 \sum_{(j_1, \ldots, j_{m}) \in \mathbb{Z}^{m}} \|U[j_1, \ldots, j_{m-1}]f  \ast \psi_{j_m} - U[j_1, \ldots, j_{m-1}]g  \ast \psi_{j_m}\|_2^2 \\
& \leq C^4 \sum_{(j_1, \ldots, j_{m-1}) \in \mathbb{Z}^{m-1}} \|U[j_1, \ldots, j_{m-1}]f  - U[j_1, \ldots, j_{m-1}]g  \|_2^2 \\
& \vdots \\
& \leq C^{2(m+1)} \|f-g\|_{\Lb^2(\mathcal{M})}^2.
\end{align*}
Thus, the claim is proven.
\end{proof}

\begin{corollary}
Suppose that $G$ satisfies the conditions of Theorem \ref{wavelet frame bounded in l2} and let $\{\psi_j\}_{j \in \mathbb{Z}}$ be a set of wavelets generated by using $G$ in Equation \ref{eq: wavelet definition}. Then we have
$$\|\overline{S}_2^m f\|^2_{\ell^2(\mathbb{Z}^m)} \leq C^{2m} \|f\|_{\Lb^2(\mathcal{M})}^2$$
for all $f \in \Lb^2(\mathcal{M})$ and all $m \geq 1$.
\end{corollary}

For proper invariance, we provide a theorem that demonstrates that the $2$-nonwindowed geometric scattering transform is invariant to isometries.
\begin{theorem} 
Let $\xi \in \text{Isom}(\mathcal{M}, \mathcal{M}')$, and let $f \in \Lb^2(\mathcal{M})$. Define $f' = V_\xi f$ and let $\left(\overline{S}_2^{m}\right)'$ be the corresponding $2$-nonwindowed geometric scattering transform on $\mathcal{M}'$ produced by a littlewood paley wavelet satisfying the conditions described in Theorem \ref{wavelet frame bounded in l2}. We have $\left(\overline{S}_2^{m}\right)'f' = \overline{S}_2^m f$.
\end{theorem}
\begin{proof} We see that $\overline{S}_2[\emptyset]f = \|f\|_2 = \|V_\xi f\|_2$ since $V_\xi$ is an isometry. Now suppose that we consider $p = (j_1, \ldots, j_m)$. Then since convolution using a spectral filter commutes with isometries and modulus operators (see Theorem 2.1 in \cite{geometric}), 
\begin{align*}
\overline{S}_2^m[j_1, \ldots, j_m] f &= \|U[p]f\|_{\Lb^2(\mathcal{M})} \\
&= \|V_\xi U[p]f\|_{\Lb^2(\mathcal{M})} \\
&= \|U[p] V_\xi f\|_{\Lb^2(\mathcal{M})} \\
&= \|U[p] f'\|_{\Lb^2(\mathcal{M})} \\
&= \left(\overline{S}_2^{m}\right)'[j_1, \ldots, j_m] f'.
\end{align*} 
Thus, we can see that each layer is isometry invariant.
\end{proof}

\subsection{The $q$-nonwindowed Geometric Scattering Norm}
Now we prove the $q$-nonwindowed Geometric Scattering Transforms, for $q \in (1,2)$, are bounded operators with respect to \eqref{eqn: scattering norm} under mild assumptions.

\begin{theorem} \label{q norm nonexpansive}
Suppose that we choose wavelets $\{\psi_{j}\}_{j \in \mathbb{Z}}$ generated by using $G \in \mathcal{S}(\mathbb{R}^{+})$ in Equation \ref{eq: wavelet definition}, $G$ satisfies the conditions of Theorem \ref{wavelet frame bounded in l2}, and $c_1 c_2 < 2$. Then
$$\|\overline{S}_q^mf- \overline{S}_q^m g\|^q_{\ell^2(\mathbb{Z}^m)} \leq  C_q^m\|f - g\|_{\Lb^q(\mathcal{M})}^q$$
for all $f, g \in \Lb^q(\mathcal{M})$, for all $m \geq 1$, and some constant $C_q$ dependent on $q$.
\end{theorem}
\begin{proof}
We start by providing a proof for the case of $m = 1$:
\begin{align*}
\|\overline{S}_q f- \overline{S}_q g\|^q_{\ell^2(\mathbb{Z}^m)}  &= \left(\sum_{j \in \Z} | \overline{S}_q[j]f - \overline{S}_q[j] g|^2\right)^{q/2} \\
& = \left(\sum_{j \in \Z} |\|U[j]f\|_q - \|U[j]g\|_q|^2\right)^{q/2} \\
& \leq \left(\sum_{j \in \Z} \|U[j]f - U[j]g\|_q^2\right)^{q/2} \\
&= \left(\sum_{j \in \Z} \left(\int_{\mathcal{M}} |U[j]f(x) - U[j]g(x)|^q \, d \mu(x)\right)^{2/q}\right)^{q/2}
\end{align*}
Via Minkowski's Integral Intequality, 
\begin{align*}
&\left(\sum_{j \in \Z} \left(\int_{\mathcal{M}} |U[j]f(x) - U[j]g(x)|^q \, d \mu(x)\right)^{2/q}\right)^{q/2}\\
&\leq \int_{\mathcal{M}} \left( \sum_{j \in \Z} |U[j]f(x) - U[j]g(x)|^2 \right)^{q/2} \, d \mu(x) \\
&\leq \int_{\mathcal{M}} \left( \sum_{j \in \Z} |(f \ast \psi_j)(x) - (g \ast \psi_j)(x)|^2 \right)^{q/2} \, d \mu(x) \\
&\leq \|\vec{T}(f-g)\|_{\Lb^q_{\ell^2(\mathbb{Z})}(\mathcal{M})}^q.
\end{align*} Now apply Corollary \ref{operator extension theorem} to get 
$$\|\vec{T}(f-g)\|_{\Lb^q_{\ell^2(\mathbb{Z})}(\mathcal{M})}^q \leq C_q \| f-g\|_{\Lb^q(\mathcal{M})}^q.$$
Now assume that for some $m \geq 1$, we have
$$\|\overline{S}_q^{m}f- \overline{S}_q^{m} g\|^q_{\ell^2(\mathbb{Z}^m)} \leq C_q^m\|f - g\|_{\Lb^q(\mathcal{M})}^q.$$

Similar to above, when we consider the case with $m+1$, we can mimic the steps above to get 
\begin{align*}
&\|\overline{S}_q^{m+1} f- \overline{S}_q^{m+1} g\|^q_{\ell^2(\mathbb{Z}^m)}  \\
&= \left(\sum_{j_{m+1} \in \Z} \cdots \sum_{j_1 \in \Z} | \overline{S}_q^{m+1}[j_1, \ldots, j_{m+1}]f - \overline{S}_q^{m+1}[j_1, \ldots, j_{m+1}] g|^2\right)^{q/2} \\
& = \left(\sum_{(j_1, \ldots, j_m) \in \Z^m} \cdots \sum_{j_{m+1} \in \Z} \left(\int_{\mathcal{M}} |U[j_1, \ldots, j_{m+1}]f(x) - U[j_1, \ldots, j_{m+1}]g(x)|^q \, d \mu(x)\right)^{2/q}\right)^{q/2} \\
& = \left(\sum_{(j_1, \ldots, j_m) \in \Z^m}  \left(\sum_{j_{m+1} \in \Z} \left(\int_{\mathcal{M}} |U[j_1, \ldots, j_{m+1}]f(x) - U[j_1, \ldots, j_{m+1}]g(x)|^q \, d \mu(x)\right)^{2/q}\right)^{\frac{q}{2} \cdot \frac{2}{q}}\right)^{q/2} \\
& \leq \left(\sum_{(j_1, \ldots, j_m) \in \Z^m}  \left(\int_{\mathcal{M}} \left(\sum_{j_{m+1} \in \Z} |U[j_1, \ldots, j_{m+1}]f(x) - U[j_1, \ldots, j_{m+1}]g(x)|^2 \right)^{q/2} \, d \mu(x)\right)^{2/q}\right)^{q/2} \\
& \leq  \left(\sum_{(j_1, \ldots, j_m) \in \Z^m}  \left(\int_{\mathcal{M}} \left(\sum_{j_{m+1} \in \Z} |(U[j_1, \ldots, j_m]f \ast \psi_{j_{m+1}})(x) - (U[j_1, \ldots, j_m]g \ast \psi_{j_{m+1}})(x)|^2 \right)^{q/2} \, d \mu(x)\right)^{2/q}\right)^{q/2} \\
&=  \left(\sum_{(j_1, \ldots, j_m) \in \Z^m}  \|\vec{T}(U[j_1, \ldots, j_m]f-U[j_1, \ldots, j_m]g)\|_{\Lb^q_{\ell^2(\mathbb{Z})}(\mathcal{M})}^2\right)^{q/2} \\
&=  C_q\left(\sum_{(j_1, \ldots, j_m) \in \Z^m}  \|U[j_1, \ldots, j_m]f-U[j_1, \ldots, j_m]g\|_{\Lb^q(\mathcal{M})}^2\right)^{q/2}.
\end{align*}

Now we see that we can apply the induction hypothesis to get
\begin{align*}
\left(\sum_{(j_1, \ldots, j_m) \in \Z^m}  \|U[j_1, \ldots, j_m]f-U[j_1, \ldots, j_m]g\|_{\Lb^q_{\ell^2(\mathbb{Z})}(\mathcal{M})}^2\right)^{q/2} &= \|\overline{S}_q^{m}[j_1, \ldots, j_{m}] f- \overline{S}_q^{m}[j_1, \ldots, j_{m}] g\|^q_{\ell^2(\mathbb{Z}^m)} \\
&\leq C_q^m \|f-g\|_{\Lb^q(\mathcal{M})}^q.
\end{align*}
\end{proof}

\begin{corollary} \label{q norm well-defined}
Suppose that we choose wavelets $\{\psi_{j}\}_{j \in \mathbb{Z}}$ generated by using $G \in \mathcal{S}(\mathbb{R}^{+})$ in Equation \ref{eq: wavelet definition}, $G$ satisfies the conditions of Theorem \ref{wavelet frame bounded in l2}, and $c_1 c_2 < 2$. Then
$$\|\overline{S}_q^mf\|^q_{\ell^2(\mathbb{Z}^m)} \leq  C_q^m\|f \|_{\Lb^q(\mathcal{M})}^q$$
for all $f\in \Lb^q(\mathcal{M})$, for all $m \geq 1$, and some constant $C_q$ dependent on $q$.
\end{corollary}

For the next theorem, we omit the proof since it is identical to the case when $q = 2$, but we state it for completeness.
\begin{theorem} 
Let $\xi \in \text{Isom}(\mathcal{M}, \mathcal{M}')$, and let $f \in \Lb^q(\mathcal{M})$. Define $f' = V_\xi f$ and let $\left(\overline{S}_q^{m}\right)'$ be the corresponding $q$-nonwindowed geometric scattering transform on $\mathcal{M}'$ produced by wavelets $\{\psi_{j}\}_{j \in \mathbb{Z}}$ using $G \in \mathcal{S}(\mathbb{R}^{+})$ in Equation \ref{eq: wavelet definition}, $G$ satisfies the conditions of Theorem \ref{wavelet frame bounded in l2}, and $c_1 c_2 < 2$ in Lemma \ref{manifold regularity}. We have $\left(\overline{S}_q^{m}\right)'f' = \overline{S}_q^m f$.
\end{theorem}

\section{Diffeomorphism Stability}
In this section, we provide diffeomorphism stability results for a generalization of bandlimited functions, $\lambda$-bandlimited functions, which are defined as functions which satisfy $\hat{f}(k) = \langle f, \phi_k \rangle = 0$ whenever $\lambda_k \geq \lambda$. 

\begin{lemma} [\cite{geometric}] \label{lemma: bandlimited L^2 stability}
Suppose $\xi \in \text{Diff}(\mathcal{M})$. If $f \in \Lb^2(\mathcal{M})$ is $\lambda$-bandlimited,  then 
$$\|f - V_\xi f\|_{\Lb^2(\mathcal{M})} \leq C(\mathcal{M}) \lambda^n \|\xi\|_\infty \|f\|_{\Lb^2(\mathcal{M})}$$
for some constant $C(\mathcal{M})$.
\end{lemma}






\begin{theorem} 
Suppose $\xi \in \text{Diff}(\mathcal{M})$. Let $f \in \Lb^2(\mathcal{M})$, and assume that $\psi$ is a wavelet family satisfying the conditions of Theorem \ref{wavelet frame bounded in l2}. Then  
$$\|\overline{S}_2^mf - \overline{S}_2^m V_\xi f\|_{\ell^2(\mathbb{Z}^m)} \leq C(\mathcal{M}) \lambda^n \|\xi\|_\infty \|f\|_{\Lb^2(\mathcal{M})}.$$
\end{theorem}
\begin{proof} We apply Theorem \ref{thm: q = 2 nonexpansive}, so Lemma \ref{lemma: bandlimited L^2 stability} gives the desired result.
\end{proof}

\subsection{Stability Results for the $q$-nonwindowed Geometric Scattering Norm}
\begin{lemma} \label{q norm stability}
Suppose $\xi \in \text{Diff}(\mathcal{M})$. If $f \in \Lb^q(\mathcal{M})$ is $\lambda$-bandlimited, then 
$$\|f - V_\xi f\|_{\Lb^q(\mathcal{M})}\leq C(\mathcal{M}) \lambda^n \|\xi\|_\infty \|f\|_{\Lb^q(\mathcal{M})}$$
for some constant $C_q(\mathcal{M})$.
\end{lemma}
\begin{proof}
Since $f$ is $\lambda$-bandlimited, $f \in \Lb^2(\mathcal{M})$ as well, and the proof is nearly identical to the proof of the case when $q = 2$, but we provide the steps for completeness. We define $\pi_\lambda$ be the operator that projects a function $f \in \Lb^2(\mathcal{M})$ onto the eigenspace $E_\lambda$ and define the projection operator
$$P_\lambda := \sum_{\lambda_n \leq \lambda} \pi_{\lambda_n}$$
with kernel
$$K^{(\lambda)}(x,y) = \sum_{\lambda_n \leq \lambda} e_n(x) \overline{e_n(y)}.$$
We have $P_\lambda f = f$ $\mu$-almost-everywhere. Thus, via Holder's inequality,
\begin{align*} \label{q norm scattering stability}
|f(x)- V_\xi f(x)| &= |P_\lambda f(x) -  V_\xi P_\lambda f(x)| \\
&= \left|\int_{\mathcal{M} }K^{(\lambda)}(x,y) f(y) \, dy - \int_{\mathcal{M} }K^{(\lambda)}(\xi^{-1}(x),y) f(y) \, dy\right| \\
& \leq \left|\int_{\mathcal{M} }(K^{(\lambda)}(x,y) - K^{(\lambda)}(\xi^{-1}(x),y)) f(y) \, dy\right| \\
& \leq \|f\|_{\Lb^q(\mathcal{M})}\left(\int_{\mathcal{M} }|K^{(\lambda)}(x,y) - K^{(\lambda)}(\xi^{-1}(x),y)|^p \, dy\right)^{1/p} \\
& \leq C_{q, \text{Vol}}(\mathcal{M})\|f\|_{\Lb^q(\mathcal{M})} \|\xi\|_\infty \|\nabla K^{(\lambda)}\|_\infty
\end{align*} 
for some constant $C_{q, \text{Vol}}(\mathcal{M})$ dependent on $q$ and the volume of the manifold. Now, by Lemma H.1 in \cite{geometric}, we have
$$\|\nabla K^{(\lambda)} \|_\infty \leq C_q(\mathcal{M}) \lambda^n.$$
Thus, the proof is complete.
\end{proof}
%

\begin{theorem} 
Suppose $\xi \in \text{Diff}(\mathcal{M})$. Let $f \in \Lb^q(\mathcal{M})$ be $\lambda$-bandlimited. Additionally, suppose that we choose wavelets $\{\psi_{j}\}_{j \in \mathbb{Z}}$ generated by using $G \in \mathcal{S}(\mathbb{R}^{+})$ in Equation \ref{eq: wavelet definition}, $G$ satisfies the conditions of Theorem \ref{wavelet frame bounded in l2}, and $c_1 c_2 < 2$. Then
$$\|\overline{S}_q^mf - \overline{S}_q^m V_\xi f\|_{\ell^2(\mathbb{Z}^m)} \leq C(\mathcal{M}) \lambda^n \|\xi\|_\infty \|f\|_{\Lb^q(\mathcal{M})}$$
for some constant $C(\mathcal{M})$.
\end{theorem} 
\begin{proof} We apply Theorem \ref{q norm nonexpansive} to get
$$\|\overline{S}_q^mf - \overline{S}_q^m V_\xi f\|_{\ell^2(\mathbb{Z}^m)} \leq C_q \|f-V_\xi f\|_{\Lb^q(\mathcal{M})}.$$
By Lemma \ref{q norm stability}, we have
$$\|f-V_\xi f\|_{\Lb^q(\mathcal{M})} \leq C(\mathcal{M}) \lambda^n \|\xi\|_\infty \|f\|_{\Lb^q(\mathcal{M})},$$
which gives the desired result.
\end{proof}

\section{Conclusions and Future Work}
We have provided a framework for understanding nonwindowed scattering coefficients. In particular, we provide a weighted measure for distortion between nonwindowed scattering coefficients, showed our weighted measure is well-defined mapping for $\Lb^q(\mathcal{M})$ functions, and showed that nonwindowed scattering coefficients are stable to diffeomorphisms for $\lambda$-bandlimited functions. 

For future work, it is of interest to see if it is possible to extend our results to manifolds that are not restricted the conditions present in Sections 4 and 5. Additionally, what are other manifolds that satisfy the conditions present in sections 4 and 5? These questions will be left to future work.

\section{Acknowledgements}
We would like to thank Michael Perlmutter for providing feedback, which improved the quality of this initial draft.

\bibliographystyle{elsarticle-num}
\bibliography{paper.bib}

\end{document}